\documentclass[a4paper,11pt]{article}

\title{On the Brill-Noether loci of a curve embedded in a $\Kt$ surface}
\author{Luigi Pagano}

\usepackage{CommandsAndStyle}

\usepackage{titlesec}
\titleformat{\subsubsection}[runin]{\bfseries}{\thesubsubsection}{1em}{}[]

\begin{document}

\maketitle

\begin{abstract}
We extend a previous result of Feyzbakhsh concerning the injectivity of a map of moduli spaces and we use this result to construct curves whose Brill-Noether loci have unexpected dimension.

MSC: 14H60, 14J28, 14H51.

Keywords: Bridgeland stability, Mukai program, Brill-Noether locus.
\end{abstract}

\section{Introduction}


Let $\mathcal{P}_g$ be the moduli space of triples $(X,H,C)$ where $X$ is a projective $K3$ surface with polarization $H$ and $C$ is a smooth curve of genus $g$ in the linear system $|H|$ and let $\mathcal{M}_g$ be the moduli space of curves of genus $g$. Consider the forgetful functor $\mathcal{P}_g\rightarrow \mathcal{M}_g$, $(X, H, C)\mapsto C$. If $g\geq 11$ and $g\neq 12$, the functor is birational onto its image. In \cite{Muk96} Mukai started a program consisting in finding the rational inverse of such map, i.e. in reconstructing the $\Kt$ surface $X$ starting from the curve $C$.

Mukai's program has been dealt with in \cite{ABS14, Fey17, Fey20, Muk96}; it has been proved that when $X$ is a general polarized $\Kt$ surface of genus $g\geq 11$, if the curve $C$ belongs to the linear system of a primitive ample line bundle $H$, then $X$ is the unique $\Kt$ surface containing $C$.

All the four proofs of the results carried, as a side product, the knowledge of the structure of a specific Brill-Noether locus of the curve $C$, more precisely it has been proven that the Brill-Noether locus $T_g\coloneqq B_C^{s+2}(2,4s)$, i.e. the space parametrizing vector bundles of rank 2, degree 4s and having at least $s+2$ linearly independent global sections, if $g=2s+1$, or the locus $T_g\coloneqq B_C^{p+4}(4,4p)$ if $g=p+1$ and $p$ is odd, are smooth $\Kt$ surfaces. For $g\neq 11$, the virtual dimension of such Brill-Noether loci is negative.


Let $v\coloneqq (2, H, s)$ if $g=2s+1$ or $v\coloneqq (4, 2H, p)$ if $g=p+1$ and $p$ is odd.
Let $N_g\coloneqq M_{X, H}(v)$ be the moduli space of stable sheaves on $X$ with Mukai vector $v$.
In this paper we are going to study the morphism from the moduli space $N_g$ to $T_g$ induced by the restriction map sending a sheaf on $X$ to its restriction on $C$. We are going to study under what conditions such restriction map is injective, obtaining the following results:

\begin{theorem}
\label{thm:BN1}
Let $X$ be a $\Kt$ surface which is not hyperelliptic.
Let $C\hookrightarrow X$ be an hyperplane section of genus $g\equiv 1 \pmod 4$.
Then the map $N_g\rightarrow T_g$ sending a vector bundle $E$ on $X$ with Mukai vector $v$ to its restriction on $C$ is injective.
\end{theorem}

\begin{theorem}
\label{thm:BN23}
Let $C\hookrightarrow X$ be an hyperplane section of genus $g\equiv 2,3 \pmod 4$.
Then the map $N_g\rightarrow T_g$ sending a vector bundle $E$ on $X$ with Mukai vector $v$ to its restriction on $C$ is injective.
\end{theorem}

We are going to prove Theorem \ref{thm:BN1} in \S \ref{sect:cong1}, while the proof of Theorem \ref{thm:BN23} is split in \S \ref{sect:cong3} and \S \ref{sect:cong2} according to the case.

After that, we are going to consider a curve contained in a (possibly non unique) $\Kt$ surface and repeat the construction of such BN loci, showing that when the curve is contained a family of different $\Kt$ surfaces, its BN locus must have dimension greater than two. In Example \ref{ex:curvehBN}, by giving a codimension two linear section of a Fano threefold, we construct a curve whose BN locus has dimension at least three. By adapting this construction to different cases, it is possible to get several examples of curves whose BN loci have actual dimension higher than the expected one.



\section{Stability conditions}

\subsubsection{} In this section, we summarize the main properties of the stability notion introduced by Bridgeland in \cite{Bri07, Bri08}. For a $\Kt$ surface $X$ let us denote by $\Stab(X)$ the set of locally finite stbility conditions on $\mathcal{D}^\flat(\Coh(X))$ and by $\Stab^\dagger(X)$ the connected component of $\Stab(X)$ containing all the stability conditions constructed by Bridgeland in \cite{Bri08}.

\subsubsection{} Let $X$ be a projective $\Kt$ surface and let $\beta, \omega\in \NS_{\R}(X)\coloneqq \NS(X)\otimes \R$, with $\omega\in \Amp(X)$. Following the construction of \cite[\S 6]{Bri08}, we can associate a stability condition $\sigma_{\beta,\omega}$ to the two divisors.
Let $\mathcal{T}\subseteq \Coh(X)$ be the set of sheaves whose torsion-free parts have $\omega-$slope semistable Harder-Narasimhan factors of slope $\mu_\omega> \beta\cdot \omega$ and let $\mathcal{F}\subseteq \Coh(X)$ be the set of torsion-free sheaves whose Harder-Narasimhan factors have slope $\mu_\omega\leq \beta\cdot \omega$.
Then $(\mathcal{T,F})$ is a torsion pair, \cite[Lemma 6.1]{Bri08}, and tilting with respect to this pair gives a bounded t-structure whose heart is
$$\mathcal{A}(\beta, \omega)\coloneqq \{ E\in \mathcal{D}(X) \colon H^i(E)=0 \mbox{ for } i \neq -1,0, H^{-1}(E)\in \mathcal{F} \mbox{ and }  H^0(E)\in \mathcal{T} \} \punto $$
The stability condition $\sigma_{\beta, \omega}$ has the above heart and central charge given by:
$$ Z_{\beta, \omega}(E)= \beta\cdot\Delta - s - r \frac{\beta^2-\omega^2}2 + i (\Delta-r\beta)\cdot \omega  \virgola $$
for $E\in \mathcal{D}^\flat(X)$ with Mukai vector $v(E)=(r,\Delta,s)$. This defines a stability condition provided that for all spherical objects $E$, one has $Z_{\beta, \omega}(E)\notin \R_{\leq0}$; in particular this defines a stability condition whenever $\omega^2>2$, \cite[Lemma 6.2]{Bri08}.

\subsubsection{} If $H\in X$ is a fixed ample line bundle, we consider a special subset of $\Stab^\dagger(X)$ parametrized by the upper half plane: for $a>0, b \in \R$ we denote by $\sigma_{b,a}$ the stablity condition $\sigma_{bH, aH}$.

\subsection{Wall and chambers structure}

\subsubsection{} The vector space $\mathcal{N}_{\R}(X)\coloneqq \R\oplus \NS_{\R}(X) \oplus \R$, endowed with the Mukai pairing, has signature $(2,\rho(X)),$ where $\rho(X)=\rk \Pic(X)$ is the Picard number of $X$.

\subsubsection{} It has been proven in $\cite{Bri08}$ that the kernel of the central charge 
$$ Z_{\beta, \omega} \colon\mathcal{N}_\R(X) \rightarrow \C  $$
is a codimension two subspace of $\mathcal{N}_\R(X)$ which is negative definite.
Let $k(\beta, \omega)\in G(\rho, \rho+2)=G(\rho, \mathcal{N}_{\R}(X))$ be the point of the Grassmannian corresponding to $\ker Z_{\beta, \omega}$.
We will call $V(X)\subseteq G(\rho,\rho+2)$ the set of points corresponding to some stability condition via the construction above. Given an object $E\in \mathcal{D}(X)$ there is a locally finite set of real hypersurfaces in $V(X)$ (walls) $\{W_\gamma \colon \gamma\in \Gamma\}$ such that each chamber, i.e. a connected component
$$C\subseteq V(X)\backslash \bigcup_{\gamma\in \Gamma } W_\gamma$$
has the following property: if $E$ is semistable for some condition $\sigma\in C$, then $E$ is semistable for all the conditions in $C$, as shown in  in \cite[Proposition 9.3]{Bri08}. Up to dropping some walls, we can assume that if $E$ is semistable (resp. unstable) in some chamber $C$, then it is unstable (resp. semistable) in the chambers sharing a wall with $C$.


\begin{proposition}[Wall-chamber structure]
Each wall $\mathcal{W}$ considered above is contained in a Schubert 1-cocycle of $G(\rho,\rho+2)$; in fact, $\mathcal{W}$ is a connected component of the intersection of such 1-cocycle with $V(X)$.
\end{proposition}

\begin{proof}
For any wall $\mathcal{W}$ there is a subobject $F\hookrightarrow E$ such that $Z(E)=\lambda Z(F)$ for $\lambda\in \R$; this means that $v(E)+\lambda v(F) \in \ker Z$, therefore $\ker Z$ intersects the plane spanned by $v(E), v(F)$ in $\R^{\rho+2}$ and this defines a Schubert 1-cocycle, hence $\mathcal{W}$ is contained in such a cocycle. In order to show that the walls are whole connected components of such a 1-cocycle is enough to notice that the phase is continuous and, therefore, if there is a point of the 1-cocycle where the two objects have the same phase, then the difference of the phases varies continuously as the stability condition varies on the 1-cocycle, but the difference of phases is an even integer, so it must be 0.
\end{proof}

\subsubsection{} The following lemma gives out a connection between the classical notions of stability, e.g. Gieseker or slope-stability, and Bridgeland's notion, it is basically \cite[Lemma 2.12]{Fey17}, which was stated in the particular case where $\rho(X)=1$; the proof of our case is the same same as in \cite{Fey17}:

\begin{lemma}
\label{lem:shiftstab}
Let $H$ be an ample line bundle and $E$ a $\mu_H-$stable vector bundle with Mukai vector $v(E)=(r,\Delta, s)$. Let $\beta_0\in\nicefrac{\Delta}r+ H^\bot$ and $\omega=w H$ for some $w\in \R_+$. Then $E[1]$ is $\sigma_{\beta_0,\omega}-$stable of phase $1$.
\end{lemma}

\subsubsection{} We conclude the section introducing a tool which allows us to bound the number of global sections of a sheaf on $X$, provided that we know something about its walls:

\begin{lemma}
\label{lem:wallglosect}
Let $\sigma_{\beta, \omega}$ be a stability condition sufficiently close to the pole $T=\{\langle \exp(\beta+i\omega), v(\mathcal{O}_X) \rangle=0\}\subseteq G(\rho,\rho+2)\}$ such that $\beta\cdot \omega < 0$. Assume $E\in\mathcal{D}(X)$ is a semistable object of the same phase as $\mathcal{O}_X$ with respect to $\sigma_{\beta,\omega}$ and that $\sigma_{\beta,\omega}$ lies in the closure of a chamber where $E$ is semistable; suppose, moreover, its Mukai vector is $v(E)=(r,cH,s)$ for some primitive effective $H\in \NS(X)$ and some $c\in \Z$. Then
$$ h^0(X,E)\leq \frac{\chi(E)}2+\frac{\sqrt{(r-s)^2+c^2(H^2+4)}}{2} \punto $$
\end{lemma}

\begin{proof}
We begin by showing that $\mathcal{O}_X$ is $\sigma_{\beta,\omega}-$stable if $k(\beta,\omega)$ is sufficiently close to $T$.

Consider a point $p\in T$ such that the plane corresponding to $p$ does not contain the Mukai vector of any spherical object other than $v(\mathcal{O}_X)=(1,0,1)$; the set of points with this property is an open and dense subset of $T$. Consider an open neighbourhood $U\ni p$ which does not contain any plane passing through the Mukai vector of a spherical object different than $(1,0,1)$. If $k(\beta, \omega)$ is close enough to $T$, then we can choose $p$ and $U$ such that $k(\beta,\omega)\in U$.

Moreover, since $\Amp(X)\otimes \Q$ is dense in $\Amp(X) \otimes \R$, there is $L\in \Amp(X)$ and $\beta_0\in L^\bot$ such that $k(\beta_0,wL)\in U$ for some $w\in \R$.

Lemma \ref{lem:shiftstab} ensures that $\mathcal{O}_X$ is $\sigma_{\beta_0,\omega}-$stable whenever $\omega$ is the multiple of an integral ample line bundle. Now, if $\mathcal{O}_X$ is not $\sigma_{\beta,\omega}-$stable, then there is a wall $\mathcal{W}$ for $\mathcal{O}_X$ which intersects $U$.

If, by contradiction, there is such a wall $\mathcal{W}$ and $F\hookrightarrow \mathcal{O}_X$ is the object defining the wall, then choose a compact set $B\subseteq U$ intersecting $\mathcal{W}$; by the construction of \cite[Proposition 9.3]{Bri08}, there is a stability condition $\tau=(Z',\mathcal{P})\in B$, such that $|Z'(F)| < m_\tau(\mathcal{O}_X)$. But if we choose $U$ small enough, $\sup_{\sigma\in U}\{m_\sigma(\mathcal{O}_X)\}$ is arbitrarily small and the only object $F$ with $v(F)^2\geq -2$ and $|Z'(F)| \ll 1$ is $\mathcal{O}_X$ (or an integral translation of its), hence there cannot be such a wall $\mathcal{W}$ and $\mathcal{O}_X$ is stable.

Let us assume $c=0$, so $0<r=s$ because $\mathcal{O}_X$ and $E$ have the same phase; we have to prove that $\displaystyle h^0(E)\leq \frac{\chi(E)}2 = \frac{r+s}2= r$. 

%
%

The inequality is obvious for $r=1$: a global section of $E$ would give an isomorphism between $E$ and $\mathcal{O}_X$.

Assume the statement is true for $r-1$. If $\mathcal{O}_X\rightarrow E$ is a global section, then we complete the triangle with $\mathcal{O}_X\rightarrow E \rightarrow Q$, where $Q$ is semistable of the same phase as $E,\mathcal{O}_X$ and $v(Q)=(r-1,0,r-1)$, so by induction $h^0(Q)\leq r-1$. After applying the functor $\Hom_X(\mathcal{O}_X,\cdot)$ to the previous triangle, we get the long exact sequence
$$ \cdots\rightarrow \Hom(\mathcal{O}_X,Q[-1])\rightarrow \Hom(\mathcal{O}_X,\mathcal{O}_X) \rightarrow \Hom(\mathcal{O}_X,E)\rightarrow\Hom(\mathcal{O}_X,Q) \rightarrow \Hom(\mathcal{O}_X,\mathcal{O}_X[1]) \rightarrow \cdots \virgola  $$
but $Q[-1]$ is semistable of phase smaller than $\mathcal{O}_X$, so $\Hom(\mathcal{O}_X,Q[-1])=0$ and $\Hom(\mathcal{O}_X,\mathcal{O}_X[1])=H^1(\mathcal{O}_X)=0$, hence we get
$$ 0\rightarrow \Hom(\mathcal{O}_X,\mathcal{O}_X) \rightarrow \Hom(\mathcal{O}_X,E)\rightarrow\Hom(\mathcal{O}_X,Q) \rightarrow 0 \virgola  $$
which yield the desired result.

Let us assume, now, $c>0$.
Consider the evaluation map
$$ \Hom_X(\mathcal{O}_X,E)\otimes \mathcal{O}_X \cong \mathcal{O}_X^{h^0(E)} \rightarrow E  \punto $$

Since $\mathcal{O}_X$ is $\sigma_{\beta,\omega}-$stable, i.e. a simple object in the category of semistable objects with the same phase as $\mathcal{O}_X$, we infer that the evaluation map is injective and its cokernel $Q$ is semistable of the same phase as $\mathcal{O}_X$.

Let $Q_1,\dots,Q_n$ be the Jordan-Hölder factors of $Q$ and let $w_i=v(Q_i)$; notice that as $\sigma$ varies in a small enough neigbourhood $\sigma_{\beta,\omega}\in U \subseteq \mathcal{H}$, where $\mathcal{H}$ is the Schubert 1-cocycle of the stability conditions satisfying $\frac{Z(E)}{Z(\mathcal{O}_X)}\in \R$, the JH filtration stays the same. We are, now, going to show that all these vectors lie in plane spanned by $v(E)=(r,cH,s)$ and $v(\mathcal{O}_X)=(1,0,1)$. Suppose, by contradiction, there is one of those vectors (WLOG say $w_1$) which does not lie on that plane.
The set
$$\Phi\coloneqq \{ \sigma=(Z,\varphi)\in V(X) \colon \varphi(E)+2\Z\pi=\varphi(\bigO_X)+2\Z\pi=\varphi(Q_1)+2\Z\pi \} $$
is a Schubert 2-cocycle since it consists on the stability conditions such that $\ker \Z\subseteq \mathcal{N}_\R(X)$ intersects the planes $\langle v(E), v(\mathcal{O}_X) \rangle, \langle v(E), w_1 \rangle, \langle \mathcal{O}_X, w_1\rangle$. Hence we get a contradiction because $\codim_{V(X)}(\Phi)=2$, while the JH filtration is invariant along a real hypersurface of $V(X)$.
It follows that $\forall i\in\{1,\dots,n\}$ there are $m_i,t_i\in \Q$ such that $w_i=m_iv(\mathcal{O}_X)+t_i v(E)$. We also have the equality
$$ \sum_{i=1}^n w_i=v(E)-h^0(E)v(\mathcal{O}_X)  \punto $$

Now we can move the stability condition $\sigma_{\beta,\omega}$ closer to the point $k(\beta_0,\omega_0)\in T$ (where $\omega_0=\sqrt{\nicefrac{2}{\omega^2}}\cdot\omega$ and $\beta_0\in \omega^\bot$), staying inside the 1-cocycle containing those stability conditions such that $E,\mathcal{O}_X$ have the same phase and $\beta\cdot \omega<0$. The JH filtration of $Q$ does not change and each factor $Q_i$ still has the same phase as $\mathcal{O}_X$, hence, by permanence of sign we get
$$ \lim_{k(\beta,\omega)\rightarrow k(\beta_0,\omega_0)}\Im (Z_{\beta,\omega}(Q_i))=\lim_{\beta\rightarrow \beta_0}(t_i(cH-r \beta)-m_i\beta)\omega=t_i c H\cdot\omega\geq 0 \punto $$

Since $H$ is effective and $\omega$ is ample, $H\cdot \omega>0$, hence we infer $t_i\geq 0$. Now, if $t_i=0$, then $v(Q_i)=m_iv(\mathcal{O}_X)$, so we conclude $Q_i\cong \mathcal{O}_X$; up to reordering factors we may assume $Q_1\cong\dots\cong Q_{i_0}\cong \mathcal{O}_X$ and $t_i>0$ for $i>i_0$. Now, if $t_i>0$, we need $t_ic H\in \NS(X)$ and since $H$ is a primitive class $t_i\in \frac1c\Z$, hence $n-i_0\leq c$.

It follows, from Lemma \ref{lem:estsquare}, that
$$ {\left({v(E)-\left({h^0(X,E)+i_0}\right)v(\mathcal{O}_X) }\right) }^2={\left({ \sum_{i=i_0+1}^n w_i }\right)}^2 \geq -2c^2  \punto $$

Consider, now, the quadratic polynomial

$$ f(x)={\left({ v(E)-xv(\mathcal{O}_X) }\right)}^2+2 c^2= -2x^2+ 2x \chi(E) +v(E)^2+2c(E)^2  \virgola $$

as it is positive for $x=i_0+h^0(X,E)$ we may conclude the desired inequality:

$$ h^0(X,E)\leq i_0+h^0(X,E)\leq \frac{\chi(E)}2+\frac{\sqrt{{(r-s)}^2+c^2(2H^2+4)}}2 \punto $$
\end{proof}

\begin{lemma}
\label{lem:estsquare}
Let $E\in \mathcal{D}(X)$ be a $\sigma_{\beta,\omega}-$semistable object with a Jordan-Hölder filtration of length $n$. Then
$$ v(E)^2\geq  -2n^2 \punto $$
\end{lemma}

\begin{proof}
Consider a Jordan-Hölder filtration $0=\tilde{E}_0\subseteq\dots\subseteq \tilde{E}_n=E$ with respect to $\sigma_{0,\omega}$; its quotients $E_i=\tilde{E}_i/\tilde{E}_{i-1}$ have the same phase as $E$. 

Now notice that $v(E)^2=\sum_i v(E_i)^2 + 2\sum_{i<j}\langle v(E_i), v(E_j) \rangle$. Since $E_j$ are stable, we have $\Hom_X(E_i,E_j)=0$ if $E_i\neq E_j$ and $\Hom_X(E_i,E_i)=\C$; therefore $\langle E_i,E_j \rangle=-\hom_X(E_i,E_j)+\ext_X^1(E_i,E_j)-\ext_X^2(E_i,E_j)\geq -2 $, where the last term is due to Serre's duality, i.e. $\ext^2(E_i,E_j)=\hom(E_j,E_i)$. It follows that $v(E)^2\geq -2n^2$.
\end{proof}

\section{The Brill-Noether loci}
\label{sect:BNloc}

\subsubsection{} Throughout this section we will consider a smooth curve embedded in a $\Kt$ surface: $i\colon C \hookrightarrow X$. We will consider the restriction of coherent sheaves $E\in \Coh(X)$ to the curve, i.e. we will consider sheaves of the form $i^\ast E= E|_{C}\in \Coh(C)$.
In this section we are going to study the map $\Coh(X)\rightarrow \Coh(C)$ sending a sheaf on $X$ into its restriction on $C$.
If we fix a Mukai vector $v\in \mathcal{N}(X)$ and an ample line bundle $H\in \NS(X)$, then there is a moduli space $M_{X,H}(v)$ parametrising $H-$Gieseker semistable sheaves $E$ with $v(E)=v$. On the other hand, we can fix non-negative integers $r,d,h\in \N$ and consider the Brill-Noether locus $B_C^h(r,d)$ of slope-stable coherent sheaves $F\in \Coh(C)$ of rank $r$, degree $d$ and having at least $h$ global sections. These two spaces are algebraic varieties; during this section we will see that with a suitable choice of the parametres we introduced before, the restriction of the map $\Coh(X)\rightarrow \Coh(C)$ to these algebraic varieties gives a morphism $\psi\colon N\coloneqq M_{X,H}(v)\rightarrow T\coloneqq B_C^h(r,d)$.

\subsubsection{} We are going to face the problem by studying the stability, in Bridgeland's sense, of sheaves in $N$ and in $i_\ast T$. This strategy has two advantages: firstly in many cases it is possible to understand a link between slope/Gieseker stability of our objects and their stability in Bridgeland's sense; on the other hand many good property of the restriction map can be deduced via these techniques. The following lemma pushes towards that direction, as it gives an useful relation between the slope stability notion for a vector bundle on $C$ and the Bridgeland stability notion for its push-forward on $X$:

\begin{lemma}
\label{lem:pushingstab}
Let $F$ a vector bundle on the curve $C$. If $F$ is slope-(semi)stable, then for any $\omega\in \Amp(X)$ there exists $\lambda_0\in \R_+$ such that $\forall \beta\in \NS(X)\otimes \R$, $\forall \lambda>\lambda_0$, we have that $i_\ast(F)$ is $\sigma_{\beta,\lambda\omega}-$(semi)stable.

Conversely, if there exists a stability condition $\sigma_{\beta,\omega}$ with respect to which $i_\ast F$ is (semi)stable, then $F$ is slope-(semi)stable.
\end{lemma}

\begin{proof}
Any coherent sheaf with rank 0 always belongs to the heart $\mathcal{A}=\mathcal{A}(\beta, \omega)$, $\forall \beta, \omega$; in particular $i_\ast F\in \mathcal{A}$ and it is enough to study the phase of subobjects of $i_\ast F$ in $\mathcal{A}$.

Now, $i_\ast F'$ is a subobject of $i_\ast F$ (in $\Coh(X)$) if and only if $F'\subseteq F$ (in $\Coh(C)$) and the inequality between their phases becomes equivalent to the inequality between their slopes. This proves the second part of the lemma, while for the first part we have only to study sheaves with positive rank.

Notice that $\Re Z_{\beta,\lambda\omega}(i_\ast F)$ is constant when $\lambda$ varies, while $\Im Z_{\beta,\lambda\omega}(i_\ast F)$ is increasing (and unbounded) in $\lambda$, therefore $\displaystyle\lim_{\lambda\rightarrow +\infty} \varphi_{\beta,\lambda\omega}(i_\ast F)= \frac12$.

If $\rk(E)>0$ we look at the terms $\rk(E)\omega^2$ in $\Re Z_{\beta,\omega}(E)$ and $\rk(E)\beta\omega$ in $\Im Z_{\beta,\omega}(E)$ and conclude that $\displaystyle\lim_{\lambda\rightarrow+\infty} \varphi_{\beta,\lambda\omega}(E)=0$. Then the phase of every positive rank subobject of $i_\ast F$ is eventually smaller than that of $i_\ast F$. To reach our conclusion we use \cite[Theorem 3.11]{Mac14} and conclude that for $\lambda>\lambda_0$ the stability condition $\sigma_{\beta,\lambda\omega}$ varies inside the same chamber.
\end{proof}

\begin{remark*}
Despite the real $\lambda_0$ of the previous result depends on $\omega$, the stability conditions obtained for $\lambda>\lambda_0$ must lie in the same chamber, we will call it \emph{Gieseker chamber}.
\end{remark*}

\subsubsection{} We will study separately some cases, depending on the remainder of the genus $g(C)$ when divided by 4, in each case the Mukai vector $v$ will be a primitive vector such that $v^2=0$, in this situation the moduli space $M_{X,H}(v)$ is a $\Kt$ surface, for generic $H$, as explained in \cite[Corollary 3.5]{Huy15}.

\subsection{The case $g\equiv 1 \pmod 4$}
\label{sect:cong1}

\subsubsection{} For the first case let us consider a primitive ample line bundle $H\in \NS(X)$. Suppose $C\in |H|$ is a projective curve of genus $g(C)=2s+1$ for some even integer $5 \leq s \in \N$. Assume, furtherly, that $X$ is not an hyperelliptic $\Kt$ surface of genus $\frac{s}2$.
Fix $v=(2,H,s)$, then $v^2=H^2-4s=2g(C)-2-4s=0$ and let $N\coloneqq M_{X,H}(v)$. Consider, then, the Brill-Noether locus $T\coloneqq B_C^{2+s}(2,4s)$. We will see that the restriction map $\psi\colon N\rightarrow T$ is well defined and injective, thus proving Theorem \ref{thm:BN1}.

\subsubsection{} Let $F\in T$ be a vector bundle; its pushforward $i_\ast F\in \Coh(X)$ has Mukai vector $\overline{v}\coloneqq v(i_\ast F)=(0, 2H, 0)$. 

\subsubsection{} Let $E\in N$. The exact sequence in $\Coh(X)$:
$$ 0\rightarrow E(-H) \rightarrow E \rightarrow i_\ast E|_C \rightarrow 0 \virgola $$
determines a triangle
$$ E\rightarrow i_\ast E|_C \rightarrow E(-H)[1] $$
in $\mathcal{D}(X)$.
We will study the stability of $E$ and $E(-H)$ with respect to stability conditions of the form $\sigma_{0,wH}$ for $w\in \R$ such that $w^2H^2>2$.

\begin{remark}
\label{subinab}
$E$ is $H-$stable of positive slope, hence it belongs to $\mathcal{T}(0,wH)$. Consider a subobject $F\hookrightarrow E$ in the abelian category $\mathcal{A}=\mathcal{A}(0,wH)$ and the corresponding exact sequence in $\mathcal{A}$:
$$ 0\rightarrow F \rightarrow E \rightarrow Q \rightarrow 0 \punto $$
This yields a long exact sequence in $\Coh(X)$:
$$ 0\rightarrow H^{-1}(Q)\rightarrow F \rightarrow E \rightarrow H^0(Q) \rightarrow 0 \virgola $$
with $H^{-1}(Q)\in \mathcal{F}$ and $H^0(Q)\in \mathcal{T}$.

Hence we may conclude that $F$ is an extension of $H^{-1}(Q)\in \mathcal{F}$ and a subsfeaf of $E$.

With a similar argument, one can infer that a quotient (in $\mathcal{A}$) of $E(-H)[1]\in \mathcal{F}[1]$ is $Q[1]$: the shift of $Q\in \mathcal{F}$. Moreover $Q$ is an extension of a quotient $Q_{-1}$ of $E$ (in $\Coh(X)$) and an object $Q_0\in \mathcal{T}$.
\end{remark}

\begin{lemma}
\label{lem:stabEA}
$E$ is $\sigma_{0,wH}-$stable $\forall\, \displaystyle w>\sqrt{\frac2{H^2}}$, unless one of the following cases happens:
\begin{enumerate}[label=(\alph*)]
\item There exists a line bundle $A_1\hookrightarrow E$ such that $A_1\cdot H= 2s-1$ and $A_1^2=s-2$;
\item There exists a line bundle $A_2\hookrightarrow E$ such that $A_2\cdot H= 2s-2$ and $A_2^2=s-2$; in this case either $(X,H)$ or $(X,H-A_2)$ would be a polarised hyperelliptic $\Kt$ surface of genus $g$, or $\displaystyle \frac{s}2+2$. 
\end{enumerate}


In both of the previous cases, $E$ is stable if $w^2H^2>s$ and it becomes semistable when $w^2H^2=s$.
\end{lemma}

\begin{remark*}
Requiring that $A_i$ as in (a) or (b) is a subsheaf of $E$ is redundant. In the proof of the proposition we will compute the Mukai vector of $A_i$ and we will notice that the Mukai pairing $-\hom_X(A_i,E)+\ext^1_X(A_i,E)-\hom_X(E,A_i)=\langle v(E), v(A_i) \rangle$ is negative and moreover $\hom(E,A_i)=0$ because $\mu_H(A_i)<\mu_H(E)$, therefore we deduce that $\hom_X(A,E)>0$.
\end{remark*}

\begin{proof}
By \cite[Proposition 14.2]{Bri08} we know that there are $w\in \R$ arbitrairily large such that $E$ is $\sigma_{0,wH}-$stable. Now suppose, by contradiction, that the statement is false. This means that there is $\displaystyle w_0>\sqrt{\frac2{H^2}}$ such that $E$ is $\sigma_{0,w_0H}-$semistable but not stable. Let $F\hookrightarrow E$ a stable subobject of $E$ (in $\mathcal{A}$) of the same phase of $E$ and let $v(F)=(r_0, \Delta_0, s_0)$ be its Mukai vector. Consider the decomposition $\Delta_0 = c_0H+\Omega$, with $\displaystyle c_0=\frac{k}{4s}\in \frac{\Z}{H^2}$ and $\Omega\in H^\bot$.

Now, $F$ is extension of a sheaf $F_{-1}\in \mathcal{F}$ and a subsheaf $F_0$ of $E$, therefore we have $\mu_H(F_{-1})\leq 0$ (or $F_{-1}=0$) and $\mu_H(F_0)<\mu_H(E)=\frac{H^2}{2}$ (or $F_0\in\{0, E\}$); note also that $F$ must be a proper subobject of $E$, therefore at least one among $F_{-1}$ and $F_0$ must be nontrivial (here we consider both $0$ and $E$ as trivial subobjects of $E$). Since $\frac{c_0H^2}{r_0}=\mu_H(F)\leq\max\{\mu_H(F_{-1}),\mu_H(F_0)\}$ the inequality $2c_0< r_0$ holds, moreover we have $c_0<1$.

The two complex numbers
$$ Z(E)=-s+w_0^2H^2 + i w H^2 $$
and
$$ Z(F)=-s_0+r_0\frac{w_0^2H^2}2 + i wc_0H^2  $$
have the same phase, in particular $c_0>0$, otherwise $F\notin \mathcal{P}(\varphi(E))$, therefore
$$ -s_0+r_0\frac{w_0^2H^2}2=c_0\left({-s+w_0^2H^2}\right) $$
which means that $\displaystyle s_0-c_0s=\frac{w_0^2H^2}{2}(r_0-2c_0)>0$ and in particular $s_0r_0>4c_0^2s$.

Moreover $F$ is stable, this means that $v(F)^2\geq-2$, therefore
$$ c_0^2H^2+\Omega^2-2r_0s_0\geq -2 \virgola $$
which leads to the inequality $r_0s_0\leq \frac12c_0^2H^2+1+\frac12\Omega^2< \frac12c_0^2H^2+1=2c_0^2s+1$.

Since $2c_0^2s<r_0s_0<2c_0^2s+1$, we have that $r_0$ must be the smallest integer strictly greater than $2c_0=\frac{k}{2s}$ and $s_0$ the smallest integer greater than $c_0s=\frac{k}4$. We can say, then, that $\displaystyle s_0=\frac{k+a}4$ and $\displaystyle r_0=\frac{k+b}{2s}$ for some integers $1\leq a \leq 4$ and $1\leq b \leq 2s$. It follows that 
\begin{equation}
\label{ineq:destabA}
 r_0s_0=\frac{k^2+(a+b)k+ab}{8s}<\frac{k^2+8s}{8s}=2c_0^2s+1 \punto 
\end{equation}

It follows that $E$ can be destabilized only if there are integers $1\leq a\leq 4, 1\leq b\leq 2s, 1\leq k \leq 4s$ satisfying the inequality \eqref{ineq:destabA} above and such that
$$
\left\{ {
\begin{array}{ll}
k\equiv -a  & \pmod 4 \\
k\equiv -b  & \pmod{2s} \punto
\end{array}
} \right.
$$

Either $k=2s-b$ or $k=4s-b$. In the second case, inequality \eqref{ineq:destabA} becomes
$$ b^2-4bs+4(2-a)s >0 \punto $$
Now, if $2\leq b\leq 2s$ the above inequality cannot be satisfied, since $b^2-4bs+4s<0$.
For $b=1$ we have also $a=1$ and the inequality is satisfied. In this case we would have $r_0=2$, $s_0=s$ and $4sc_0=\Delta_0\cdot H=4s-1$. Therefore $v(F)^2=\frac{16s^2-8s+1}{4s}+\Omega^2-4s=-2+\frac1{4s}+\Omega^2$, which means $\Omega^2=-\frac1{4s}$; the divisor $D=H-\Delta_0$ is such that $D^2=0$ and $D\cdot H=1$, the forthcoming Lemma \ref{lem:noellint1} shows that there are no such divisors in $\NS(X)$, thus there is no such an object destabilizing $E$.

If $k=2s-b$, the inequality becomes
$$ b^2-2bs+2(4-a)s>0 $$
Since we have that $s_0-c_0s=\frac{w_0^2H^2}2(r_0-2c_0)>r_0-2c_0$ (and so $\frac{a}4>\frac{b}s$), $b$ cannot be equal to $2s-1,2s-2,2s-3$. The only possible solution can be $b=1,2$ (and $a=1,2$ respectively).

For $b=1$ we must have $\displaystyle v(F)=\left({ 1, \frac{2s-1}{4s}H+\Omega_1, \frac{s}2 }\right)$, moreover
$$ v(F)^2=\frac{4s^2-4s+1}{4s} +\Omega_1^2 -s = -1+\frac{1}{4s} +\Omega_1^2 $$
which necessairly means that $\Omega_1^2=-1-\frac1{4s}$ and the case (a) is reached.

For $b=2$ we have $\displaystyle v(F)=\left({ 1, \frac{s-1}{2s}H+\Omega_2, \frac{s}2 }\right)$, moreover
$$ v(F)^2=\frac{s^2-2s+1}{s} +\Omega_2^2 -s = -2+\frac{1}{s} +\Omega_2^2 $$
which necessairly means that $\Omega_2^2=-\frac1{s}$ and we get the case (b). Let us call
$$ A_2=\frac{s-1}{2s}H+\Omega_2 $$
we have that $H-2A_2$ is an elliptic curve and that $H-A_2$ is an ample divisor, the reason for the latter statement is that $H$ is ample, so for any effective divisor $\Gamma$ we have either $\Gamma \cdot A_2\geq 0$ or $\Gamma \cdot (H-A_2)>0$, moreover $H-2A_2$ is nef, so $\Gamma\cdot (H-A_2)\geq \Gamma \cdot A_2$ and so $\forall \Gamma$ effective divisor we have that $\Gamma \cdot (H-A_2)>0$; since $(H-A_2)\cdot (H-2A_2)=2$, we infer that smooth curves in $H-A_2$ are hyperelliptic curves, hence $X$ is an hyperelliptic $\Kt$ surface with respect to the polarization $H-A_2$ (of genus $\nicefrac{s}{2}+2$), see \cite{Rei76} for more details.

In both these cases we would have $w_0^2H^2=s$.

\end{proof}

\begin{lemma}
\label{lem:noellint1}
There is no element $D\in \NS(X)$ with $D^2=0$ and $D\cdot H=1$.
\end{lemma}

\begin{proof}
Assume there is a divisor $D$ with that property.
Since $H$ is ample and $(-D)\cdot H=-1$, we have that $-D$ is not effective and $H^0(X, -D)=H^2(X, D)=0$. So $h^0(D)=h^1(D)+2\geq 2$ and $D$ is effective.

Any curve in $|D|$ must be irreducible and reduced since its intersection with the ample line bundle $H$ is equal to $1$; therefore, two different curves in $|D|$ have empty intersection because $D^2=0$. It follows that $|D|$ is a base point free pencil, hence it cuts a $g^1_1$ on $C$, which is impossible.
\end{proof}

\subsubsection{} Now let us study the stability of $E(-H)[1]$. We have that $\ch_1(E(-H))=\ch_1(E)-\rk(E)H=-H$ and that $\ch_2(E(-H))=\ch_2(E)-\ch_1(E)\cdot H + \frac12\rk(E)H^2=\ch_2(E)$, thus the Mukai vector of its shift is
$$ v(E(-H)[1])=(-2, H, -s) \punto $$
Now a similar argument as the one in Lemma \ref{lem:stabEA} leads to a similar statement for $E(-H)[1]$.

\begin{lemma}
$E(-H)$ is $\sigma_{0,wH}-$stable whenever $w^2H^2>2$, with the same exceptions as in Lemma \ref{lem:stabEA}.
\end{lemma}

\subsubsection{} The proof is actually the same as in the previous lemma: the coefficients of the two Mukai vectors are the same, up to a sign. Actually it turns out that if $F\hookrightarrow E \twoheadrightarrow Q$ is a destabilizing sequence for $E$, then $F(-H)[1]\hookrightarrow E(-H)[1]\twoheadrightarrow Q(-H)[1]$ is a destabilizing sequence for $E(-H)[1]$ (in $\mathcal{A}$), vice versa a destabilizing sequence for $E(-H)[1]$ induces a destabilizing senquence for $E$.

\subsubsection{} The next theorem shows that the restriction $i^\ast E$ of a vector bundle $E\in N$ belongs to $T$, therefore the map $\psi\colon N\rightarrow T$ is well defined.

\begin{theorem}
\label{thm:injA}
Let $E\in M_{X,H}(v)$ be a $\mu_H-$stable vector bundle. Suppose neither (a) nor (b) from Lemma \ref{lem:stabEA} hold. Then
\begin{enumerate}[label=(\alph*)]
\item The restriction $E|_{C}$ is a slope stable vector bundle on $C$ and $h^0(C,E|_C)\geq 2+s$.
\item $\Hom_X(E,E(-H)[1])=0$.
\end{enumerate}

In particular the restriction map $N\rightarrow T$ is well defined and injective.
\end{theorem}

\begin{proof}
To prove (b) it is enough to notice that at the point $\tilde{\sigma}=\sigma_{0,\tilde{w}H}$ with $\tilde{w}^2H^2=s$, both $E$ and $E(-H)[1]$ are stable of the same phase, thus we get the conclusion.

Now let us prove (a). Consider the triangle:
$$ E\rightarrow i_\ast E|_C \rightarrow E(-H)[1] \virgola$$
since $E,E(-H)[1]$ are $\tilde{\sigma}-$stable of the same phase, we get that $i_\ast E|_C$ is semistable of that phase ($\mathcal{P}(\varphi)$ is an abelian category) and $E, E(-H)[1]$ are its JH factors. According to Lemma \ref{lem:stabpf} $i_\ast E|_C$ is $\sigma_{0,wH}-$stable for $w\in (\tilde{w},\tilde{w}+\epsilon)$. Any subobject of $i_\ast (E|_C)$ in $\mathcal{A}(0,wH)$ has positive rank, so its $\sigma_{0,wH}-$phases decreases as $w$ increase, therefore $i_\ast E|_C$ is stable for all $w>\tilde{w}$. In particular, this means that $E|_C$ is slope-stable: see Lemma \ref{lem:pushingstab}.

Now apply the functor $\Hom_X(\mathcal{O}_X, \cdot)$ to the triangle above and get the long exact sequence
$$\Hom_X(\mathcal{O}_X, E(-H))\rightarrow \Hom_X(\mathcal{O}_X, E) \rightarrow \Hom_X(\mathcal{O}_X, i_\ast E|_C) \rightarrow \Hom_X(\mathcal{O}_X, E(-H)[1]) \punto $$

Both $E$ and $E(-H)[1]$ are $\sigma_{0,wH}-$stable for $\displaystyle w\rightarrow \sqrt{\frac{2}{H^2}}$. Consider the $(a,b)-$plane relative to the two divisors $\beta_0=0$ and $\omega_0=H$, i.e. consider stability conditions of the form $\sigma_{bH, aH}$. It follows from  \cite[Proposition 6.22]{MS16} (in particular from assertion (7)) that there are no walls in such $(a,b)-$plane for $E$ (resp. $E(-H)[1]$) above the semicircle given by the equation $\varphi(\mathcal{O}_X)=\varphi(E)$ (resp. $\varphi(\mathcal{O}_X)=\varphi(E(-H)[1])$), hence that semicircle lie inside, or is a wall of, a chamber where $E$ (resp. $E(-H)[1]$) is stable; in the first case an open subset of the numerical wall $\mathcal{W}=\{\varphi(\mathcal{O}_X)=\varphi(E)\}$ is contained in a chamber where $E$ (resp. $E(-H)[1]$) is stable; in the second case, as explained in following \emph{Remark} \ref{rem:OXwall}, an open subset of the numerical wall $\mathcal{W}$ is adjacent to the chamber we are considering. Thus we can apply Lemma \ref{lem:wallglosect} and deduce that
$$ h^0(X,E)\leq \frac{s+2}2+\frac{\sqrt{(s-2)^2+4s+4}}2=\frac{s+2}2+\frac{\sqrt{(s+2)^2+4}}2=s+2 + \delta $$
and that
$$ h^0(X,E(-H)[1])\leq -\frac{s+2}2+\frac{\sqrt{(s-2)^2+4s+4}}2=-\frac{s+2}2+\frac{\sqrt{(s+2)^2+4}}2= \delta \virgola $$
for some $0<\delta<1$; hence we infer that $E$ has at most $s+2$ global sections and $E(-H)[1]$ has no global section.

Now we have that $2+s=\chi(E)=h^0(X,E)-h^1(X,E)+h^2(X,E)$, but $h^2(X,E)=\hom_X(E,\mathcal{O}_X)=0$ because $E$ is $\mu_H-$stable and $\mu_H(E)>0$. It follows that $h^0(X,E)= s+2$.

But $\mathcal{O}_X$ is $\sigma_{0,\tilde{w}H}-$stable of phase 0, while $E(-H)$ is stable of phases $\varphi<0$, hence $\Hom_X(\mathcal{O}_X,E(-H))=0$; by replacing these elemnets in the long exact sequence above we get:
$$0\rightarrow \Hom_X(\mathcal{O}_X, E) \rightarrow \Hom_X(\mathcal{O}_X, i_\ast E|_C) \rightarrow 0 \virgola $$
thus $h^0(C, i_\ast E|_C)=s+2$.

The injectivity of the restriction map easily follow from the uniqueness of the JH factors of $i_\ast E|_C$: if there were $E,E'$ such that $E|_C=E'|_C$, then $i_\ast E|_C$ would have a JH filtration with factors $E, E(-H)[1]$ and a filtration with factors $E', E'(-H)[1]$, hence $E\cong E'$.
\end{proof}

\subsubsection{} This conclude the proof of Theorem \ref{thm:BN1}.

\begin{lemma}
\label{lem:stabpf}
There exists $\epsilon>0$ such that for $\tilde{w}<w<\tilde{w}+\epsilon$ we have that $i_\ast E|_C$ is $\sigma_{0,wH}-$stable.
\end{lemma}

\begin{proof}
Recall that $i_\ast E|_C$ is semistable at the point $\tilde{w}$. By local finiteness of the walls (see \cite[Proposition 9.3]{Bri08}), there exist $\varepsilon>0$ such that $i_\ast E|_C$ is always semistable or always unstable under the stability conditions $\sigma_{0, wH}$, for $w\in (\tilde{w},\tilde{w}+\varepsilon)$.
Moreover, up to choosing a smaller $\varepsilon$ we may assume that the HN filtration (or JH filtration) of $i_\ast E|_C$, with respect to $\sigma_{0,wH}$, is the same for every $w\in (\tilde{w},\tilde{w}+\varepsilon)$.

Let $F$ be a stable subobject of $i_\ast E|_C$ of maximum phase, with respect to any stability condition $\sigma_{0,wH}$. As $w\rightarrow \tilde{w}$, $F$ stays semistable. By uniqueness of JH factors, the set of JH factors of $F$ at $\tilde{w}$ is contained in $\{E, E(-H)[1]\}$, but $E(-H)[1]$ is not a subobject of $i_\ast E|_C$, hence $F=E$ or $F=i_\ast E|_C$. The first case may not occur because for $w>\tilde{w}$ the inequality $\varphi_w(E)<\varphi_w(i_\ast E|_C)$ holds; while, in the second case, $F= i_\ast E|_C$ is stable at $w\in(\tilde{w},\tilde{w}+\varepsilon)$.
\end{proof}

\begin{remark}
\label{rem:OXwall}
Keep the notation introduced in the proof of \ref{thm:injA}, we drop the hypothesis on the parity of $s$ as we will need this discussion for the next section too. We will study only the stability of $E$, because for $E(-H)[1]$ and analogous argument holds.

Assume that the arc of circle in the $(a,b)-$plane corresponding to
$$ \{(a,b)\in \R^+\times \R\colon \varphi_{b,a}(\mathcal{O}_X)=\varphi_{b,a}(E) \mbox{ and } b<0 \} \virgola $$
is an actual wall for $E$ in such $(a,b)-plane$. This means that given a point $(a,b)$ on the arc of circle, there is an actual wall $\mathcal{W}_0$ (in the entire space of stability conditions) containing the point $\sigma_{bH,aH}$, such that an open subset of the wall delimits the chamber where $E$ is stable. $\mathcal{W}_0$ is given by the equation $\{\varphi(F)=\varphi(E)\}$ for some stable subobject $F\subseteq E$ (in $\mathcal{A}(bH,aH)$); our aim is to show that, up to choosing the starting point $(a,b)$ with $|b|<<1$, we have that $F=\mathcal{O}_X$ and so $\mathcal{W}_0=\mathcal{W}$, i.e. a sufficiently small open subset of $\{\varphi(\mathcal{O}_X)=\varphi(E)\}$ is an actual wall for a chamber where $E$ is stable.

Since $F$ is a subobject of $E$ in $\mathcal{A}(bH,aH)$, it is an extension of an object $F_{-1}\in \mathcal{F}(bH,aH)$ and a subsheaf $F_0\subseteq E$. Moreover, since a small arc of circumference contained in
$$ \{(a,b)\in \R^+\times \R\colon \varphi_{b,a}(\mathcal{O}_X)=\varphi_{b,a}(E) \mbox{ and } b<0 \} \virgola $$
is contained in $\mathcal{W}_0$, we can say that the Mukai vector of $F$ is of the following form:
$$ v(F)=\alpha v(\mathcal{O}_X)+ \beta v(E) + (0,N,0)=(\alpha+2\beta, \beta H + N, \alpha+ s\beta) \virgola $$
for some $\alpha,\beta\in \Q$ and $N\in H^\bot\subseteq \NS(X)\otimes \R$. Moreover $\Im Z_{b,a}(F)$ satisfies the inequality
$$ 0< \Im Z_{b,a}(F)= a\beta H^2- (\alpha+2\beta)b H^2 \virgola $$
which in turns means that
\begin{equation}
\label{ineq:betalp}
\beta> \frac{b}{a-2b} \alpha \punto
\end{equation}

Since $F$ is stable, we have that
$$ v(F)^2=N^2+\alpha^2 v(\mathcal{O}_X)^2+\beta^2 v(E)^2+2\alpha\beta \langle v(\mathcal{O}_X) , v(E) \rangle=N^2-2\alpha^2-2(s+2)\alpha\beta \geq -2 \puntovirgola $$
it follows that 
\begin{equation}
\label{ineq:final}
\alpha(\alpha+(s+2)\beta)\leq 1 +\frac12 N^2\leq 1 \punto
\end{equation}

Now, since $\alpha+2\beta=\rk(F)$ and since $F$ is extension of two torsion-free sheaves (which cannot be simultaneously $0$) it follows that $1\leq \alpha+2\beta\in\Z$. Moreover $(s-2)\beta = \ch_2(F)\in \Z$. Finally notice that $\ch_1(F)\cdot H=4s \beta$, hence $\gcd(8,s-2)\beta \in \Z$.

Let us study the inequality \eqref{ineq:final} assuming that $\beta<0$; in particular $\beta \leq -\frac18$.
Since $a$ is bounded from below, up to choosing $|b|$ small enough, by inequality \eqref{ineq:betalp}, we may assume that $\alpha>8(s-2)\beta$, which in particular implies that inequality \eqref{ineq:final} cannot hold true.

If $\beta=0$, \eqref{ineq:final} becomes $\alpha^2\leq 1+ \frac12 N^2$ which means that $\alpha=1$ and $N=0$, since $\alpha\in \N$, so $v(F)=v(\mathcal{O}_X)$ and so $F=\mathcal{O}_X$.

Let us assume, finally, that $\beta>0$. Since $F$ is extension of a sheaf $F_{-1}\in \mathcal{F}(bH,aH)$ whose $H-$slope is negative and a subsheaf of $E$ whose slope is less or equal than $\mu_H(E)$, then the slope of $F$ is less or equal than $\mu_H(E)$, with equality if and only if $F_{-1}=0$ and $F_0=E$, i.e. iff $F=E$ which cannot happen because $F$ is a proper subobject of $E$. Therefore
$$ \frac{\beta H^2}{\alpha+2\beta}=\mu_H(F)<\mu_H(E)=\frac{H^2}2 $$
and so $\alpha>0$.

We, then, rewrite \eqref{ineq:final} as
$$ \alpha(\rk F+ s\beta)\leq 1 \punto $$

If $s$ is odd, then $\gcd(8,s-2)=1$ and $\beta \in \Z$, as well as $\alpha=\rk F-2\beta\in \Z$ and the previous inequality cannot hold true.

If $4 | s$, then $\gcd(8,s-2)=2$ and again $\alpha\in \Z$, which implies that \eqref{ineq:final} is false.

If $s\equiv 6 \pmod 8$, then $\beta\in \frac14\Z$, $\alpha\in \frac12 Z$ and $s\geq 6$, it follows that
$$ \alpha(\rk F+s\beta)\geq \frac12\left({1+\frac{s}4}\right)>1 \punto $$ 

If $s\equiv 2 \pmod 8$, then $\beta\in \frac18 \Z$, $\alpha\in \frac14\Z$ and $s\geq 10$. 
If $\beta=\frac18$, then
$$ \alpha(\rk F+s\beta)\geq \frac34\left({1+\frac{s}8}\right)>\frac32>1 \punto $$ 
If $\beta=\frac14$, then
$$ \alpha(\rk F+s\beta)\geq \frac12\left({1+\frac{s}4}\right)>\frac32>1 \punto $$
Otherwise $\beta \geq \frac38$ and 
$$ \alpha(\rk F+s\beta)\geq \frac14\left({1+\frac{30}8}\right)=\frac{38}{32}>1 \punto $$


It follows that a small open subset of the numerical wall $\{\varphi(\mathcal{O}_X)=\varphi(E)\}$ is an actual wall for a chamber where $E$ is stable.
\end{remark}

\subsubsection{} It is worth studying what happens in the case (a) of Lemma \ref{lem:stabEA}.

\begin{proposition}
\label{prop:exceptionA}
Suppose $\displaystyle A=\frac{2s-1}{4s}H+\Omega\in \NS(X)\otimes \R$ (with $\Omega\in H^\bot$) is actually a line bundle on $X$.
Assume also that case (b) in Lemma \ref{lem:stabEA} does not hold. Then there is no vector bundle $E'\in M_{X,H}(v)$ (different than $E$) such that its restrictions on $C$ coincides with the restriction of $E$, i.e. it cannot happen that $E|_C=E'|_C$.

Moreover $E|_C$ is slope semistable and $h^0(C, E|_C)\geq s+2$.
\end{proposition}

\begin{proof}
Let $A\rightarrow E \rightarrow Q$ be the destabilizing sequence for $E$ at $\sigma_{0,w_0H}$ with $w_0^2H^2=s$. We must have $\displaystyle v(Q)=\left({ 1, H-A, \frac{s}2 }\right)$; notice that, for $x\in X$, $v(\mathcal{I}_x(H-A))=v(Q)$ and that $v(Q)^2=0$, so$\dim M_{X,H}(v(Q))=2$.
It follows that $\exists x_0\in X$ such that $Q= I_{x_0}(H-A)$ because the fibre of the map $X\rightarrow M_{X,H}(v(Q))$ has $0-$dimensional fibre.

There are no non-trivial maps $Q\rightarrow A$, because
$$ \hom_X(\mathcal{I}_{x_0}(H-A),A)=h^0(X, \mathcal{I}_{x_0}^\vee (2A-H))=0 \virgola $$
where the last equality holds since $\mathcal{I}_{x_0}^\vee=\mathcal{O}_X$ and $2A-H$ is not effective (it has negative intersection with $H$). If $Q$ is a stable object there are no maps $A\rightarrow Q$, therefore
$$ \ext_X^1(A,Q)=-\hom_X(A,Q)+\ext_X^1(A,Q)-\hom_X(Q,A)=\langle v(A) , v(Q) \rangle= 1 \virgola $$
this means that any non-split extension
$$ 0\rightarrow A\rightarrow \tilde{E} \rightarrow Q \rightarrow 0 $$
must be isomorphic to $E$. Moreover, in this case, the JH factors of $i_\ast E|_C$ are $A, \mathcal{I}_{x_0}(H-A), (A-H)[1], \mathcal{I}_{x_0}(-A)[1]$, the latter two have been obtained by transforming $A, \mathcal{I}_{x_0}(H-A)$ under the functor $- \otimes (-H)[1]$.

Given any object $E'\in \mathcal{A}=\mathcal{A}(0,\omega)$ whose restriction to $C$ is $E'|_C=E|_C$, we get the exact sequence (in $\mathcal{A}$)
$$ 0\rightarrow E' \rightarrow i_\ast E|_C \rightarrow E'(-H)[1] \rightarrow 0 \punto $$
Since $E'$ is $\sigma_{0,w_0H}-$semistable of phase $\frac12$ (like $E$ and $i_\ast E|_C$), then also $E'(-H)[1]$ must be semistable of the same phase. If $E'$ has a JH filtration with factors $A_1,\dots, A_s$ and $E'(-H)[1]$ has a JH filtration with factors $B_1,\dots, B_t$, then $i_\ast E$ has a JH filtration with factors $A_1,\dots, A_s, B_1, \dots, B_t$; since JH factors of an object in $\mathcal{A}$ are uniquely determined up to order, then we have that $\{A_1,\dots,A_s,B_1,\dots B_t\}=\{ A, \mathcal{I}_{x_0}(H-A), (A-H)[1], \mathcal{I}_x(-A)[1] \}$ (the equality is intended in terms of multiset). $A$ is certainly a subobject of $E'$ and $v(E')-v(A)=v(\mathcal{I}_{x_0}(H-A))$. Since the JH factors of $E'$ are in the set $\{ A, \mathcal{I}_{x_0}(H-A), (A-H)[1], \mathcal{I}_{x_0}(-A)[1] \}$, $E'$ must be extension of $A$ and $\mathcal{I}_{x_0}(H-A)$.
But as we noticed before, there is only one non-split extension of $A$ and $\mathcal{I}_{x_0}(H-A)$, while any split extension cannot be $\mu_H-$stable, then $E'\cong E$.

If $\mathcal{I}_{x_0}(H-A)$ is not $\sigma_{0,w_0H}-$stable, then it has a subobject $L\rightarrow \mathcal{I}_{x_0}(H-A)$ which is stable of the same phase as $\mathcal{I}_{x_0}(H-A)$. Its Mukai vector satisfies $v(L)^2\geq -2$ because $L$ is stable; moreover $L$ is an extension of a subsheaf of $\mathcal{I}_{x_0}(H-A)$ and a torsion free sheaf with non-positive $H-$slope. Let $v_L=(r_L, \Delta_L, s_L)$ be its Mukai vector. The realtion $2s_L= s r_L$ holds because $\displaystyle L\in \mathcal{P}\left({\frac12}\right)$. Moreover, since $L$ is an extension of a subsheaf of $\mathcal{I}_{x_0}$ and an object with non-positive slope, the inequality $\Delta_L\cdot H\leq 2s+1$ holds; if $r_L>1$, then $v_L^2\leq \frac{(2s+1)^2}{4s}-4s<-2$, which cannot occur. Hence $r_L=1$ and $L$ is a subsheaf of $\mathcal{I}_{x_0}(H-A)$. The inequality above can be satisfied only if either $L=A$ or $L$ is such that 
$$ v(L)=\left({ 1, \frac{H}2+N, \frac{s}2 }\right) \virgola$$
for some $N\in H^\bot$; in such case we must have $N^2=-2$ since $N=0\Rightarrow H\in 2\NS(X)$ and $N^2<-2\Rightarrow v(L)<-2$. (There is another numerical possibility for $v(L)$, but we excluded it together with case (b) of Lemma \ref{lem:stabEA}).

In the first case, $A$ is a subobject of $\mathcal{I}_x(H-A)$, thus $H^0(X,\mathcal{I}_x(H-2A))\neq 0$, in particular $H-2A$ is effective, but $(H-2A)\cdot H=2$ and $(H-2A)^2=-4$, this means that an effective divisor $\Gamma \in \lvert H-2A \rvert$ cannot be integral (i.e. an irreducible and reduced subscheme of $X$), thus $\Gamma$ has at least two irreducible and reduced components $\Gamma_1,\Gamma_2$. Since $H$ is ample, $\Gamma_i\cdot H\geq 1$, hence $\Gamma=\Gamma_1+\Gamma_2$.
Moreover $\Gamma_i^2\geq -2$ and $\Gamma^2=-4$, then both $\Gamma_1,\Gamma_2$ are rational curves and none of the linear system $\lvert \Gamma_i \rvert$ has more than a curve. Moreover one can check that $\Gamma_i\cdot A\in \{0,1\}$, therefore, setting $H-A=A+\Gamma_1+\Gamma_2$ we get a contradiction:
$$ (H-A)^2=A^2-4+2A(\Gamma_1+\Gamma_2)\leq s-6+4=s-2<s=(H-A)^2 \punto $$

If $L\hookrightarrow \mathcal{I}_x(H-A)$ with $\displaystyle v(L)=\left({ 1, \frac{H}2+N, \frac{s}2 }\right)$, then with a similar argument as before we conclude that that $\mathcal{I}_{x_0}\left({\frac1{4s}H-(\Omega+N)}\right)$ has non-zero global sections, in particular $B\coloneqq \frac1{4s}H-(\Omega+N)$ is effective. Then we have that $B\cdot H=1$, which means that curves in $|B|$ are integral. Hence $-2\leq B^2=\frac1{4s}+(\Omega+N)^2$, which means that $(\Omega+N)^2=-2-\frac1{4s}$: if $(\Omega+N)^2=-\frac1{4s}$, then according to Lemma \ref{lem:noellint1} such $B$ cannot exists.
It follows that $|B|$ consists in a unique rational curve and $\mathcal{I}_{x_0}(B)$ cannot have non-trivial global sections: $H^0(X, I_{x_0}(B))$ is the kernel of the restriction map $H^0(X, B)\rightarrow H^0(x_0, \mathcal{O}_{x_0})$, but $H^0(X, B)\cong \C$ and the map sends nonzero global sections to their evaluations in $x_0$.

Thus $\mathcal{I}_x(H-A)$ must be $\sigma_{0,w_0H}-$stable and the first part is proved.

Now consider the exact sequence $0\rightarrow E(-H)\rightarrow E \rightarrow i_\ast E|_C \rightarrow 0$; the long exact sequence induced in cohomology tell us that $h^0(C, E|_C)\geq h^0(X, E)$ because $E(-H)$ has no global sections. Since $h^0(X,E)-h^1(X,E)+h^2(X,E)=\chi(E)=s+2$ and since $h^2(X,E)=\hom_X(E, \mathcal{O}_X)=0$ because $E,\mathcal{O}_X$ are $\mu_H-$stable and $E$ has positive slope, we infer that
$$ h^0(C,E|_C)\geq h^0(X, E)= s+2+h^1(X,E)\geq s+2 \punto $$
The semistability of $E|_C$ follows from Lemma \ref{lem:pushingstab} since $i_\ast E|_C$ is $\sigma_{0,\tilde{w}}-$semistable.
\end{proof}


\begin{remark}
\label{rem:morph}
Let us keep the same notations introduced in \ref{prop:exceptionA}; let us consider the product $X\times X$ and its two projections $p_1,p_2\colon X\times X \rightarrow X$ and let $\Delta\subseteq X\times X$ be the diagonal of the product, whose ideal sheaf is denoted by $\mathcal{I}_\Delta$. We obtain the sheaf $\mathcal{I}_\Delta(p_1^\ast(H-A))$ whose restriction to $X\times \{x\}\subseteq X\times X$ is $\mathcal{I}_x(H-A)$ for any point $x\in X$.

This construction gives a morphism $\rho\colon X\rightarrow Z\subseteq M_{X,H}(v(Q))$ by $\rho(x)\coloneqq\mathcal{I}_x(H-A)$ whose image $Z$ is closed because $X$ is projective.
\end{remark}

\subsection{The case $g\equiv 3 \pmod 4$}
\label{sect:cong3}

\subsubsection{} When the genus is congruent to 3 mod 4, i.e $g=2s+1$ for some odd integer $s\geq 5$, then all the settings stay unchanged: we still choose $v=(2,H,s)$ and $T=B_C^{s+2}(2,4s)$. In this paragraph we prove that the restriction map $N\rightarrow T$ is injective also in this case, hence the Theorem \ref{thm:BN23} holds in the case with $g\equiv 3 \pmod 4$.

\begin{lemma}
\label{lem:stabEB}
Assume that the gonality of $C$ is at least 7 and that $X$ is not an hyperelliptic $K3$ surface.

$E$ is $\sigma_{0,wH}-$stable $\forall\, \displaystyle w>\sqrt{\frac2{H^2}}$.
\end{lemma}

\begin{proof}
As we did in the proof of the previous case, we get the same inequality
\begin{equation}
\label{ineq:destabB}
 r_0s_0=\frac{k^2+(a+b)k+ab}{8s}<\frac{k^2+8s}{8s}=2c_0^2s+1 \punto 
\end{equation}

$E$ can be destabilized only if there are integers $1\leq a\leq 4, 1\leq b\leq 2s, 1\leq k \leq 4s$ satisfying the inequality \eqref{ineq:destabB} above and such that
$$
\left\{ {
\begin{array}{ll}
k\equiv -a  & \pmod 4 \\
k\equiv -b  & \pmod{2s} \punto
\end{array}
} \right.
$$

Either $k=2s-b$ or $k=4s-b$. In the second case, Inequality \eqref{ineq:destabB} becomes
$$ b^2-4bs+4(2-a)s >0 \punto $$
Now, if $2\leq b\leq 2s$ the above inequality cannot be satisfied, since $b^2-4bs+4s<0$.
For $b=1$ we have also $a=1$ and the inequality is satisfied. In this case we would have $r_0=2$, $s_0=s$ and $4sc_0=\Delta_0\cdot H=4s-1$. As in the Lemma \ref{lem:stabEA}, there cannot be any such object destabilizing $E$.

If $k=2s-b$, Inequality \eqref{ineq:destabB} becomes
$$ b^2-2bs+2(4-a)s>0 $$
Since we have that $s_0-c_0s=\frac{w_0^2H^2}2(r_0-2c_0)>r_0-2c_0$ (and so $\frac{a}4>\frac{b}s$), $b$ cannot be equal to $2s-1,2s-2,2s-3$. Also $4\leq b\leq 2s-4$ does not carry any solution. The only possible solution can be $b=1,3$ (and $a=3,1$ respectively). 

For $b=1,a=3$ we have that $\displaystyle v(F)=\left({1,\frac{2s-1}{4s}H +\Omega, \frac{s+1}2 }\right)$. Then $v(F)^2=-2+\frac{1}{4s}+\Omega^2$, which implies $\Omega^2=\frac1{4s}$. Consider the divisor $D=H-2\Delta_1$, where $\Delta_1=c_1(F)$, then $D^2=0$ and $D\cdot H=2$, it follows, from an argument similar to that one in the proof of Lemma \ref{lem:noellint1}, that $D$ must be an effective divisor; curves in $|D|$ are elliptic. Let $R$ be an effective divisor; $R\cdot H>0$, thus either $R\cdot \Delta_1>0$ or $R\cdot (\Delta_1+D)>0$. Since $D$ contains elliptic curves, $R\cdot D\geq 0$, thus $R\cdot (D+\Delta_1)>0$. This means that $D+\Delta_1\in \Amp(X)$
Notice that $D\cdot \Delta_1=1$, therefore $D$ cut a $g_1^1$ on a smooth curve belonging to $|\Delta_1+D|$, which is impossible. Therefore this case will never occur.

For $b=3,a=1$ we have that $\displaystyle v(F)=\left({1,\frac{2s-3}{4s}H +\Omega, \frac{s-1}2 }\right)$. Then $v(F)^2=-2+\frac{9}{4s}+\Omega^2$, which implies $\Omega^2=-\frac9{4s}$, consider the divisor $D=H-2\Delta_2$, where $\Delta_2=c_1(F)$, then $D^2=0$ and $D\cdot H=6$; it follows that such a divisor cut out a $g^1_6$ on the general curve of $|H|$ (better say any smooth curve of $H$ not containing base points of $|D|$), hence there cannot exist such a divisor if the gonality of $C$ (which is greater or equal than the gonality of the generic curve in $|H|$) is at least 7.
\end{proof}

\begin{remark}
If we drop the hypothesis on the gonality of $C$ being at least 7, then it will be possible that an object $F$ with Mukai vector $\displaystyle v(F)=\left({1,\frac{2s-3}{4s}H +\Omega, \frac{s-1}2 }\right)$ destabilizes $E$, but in this case we would have
$$ \frac{a}2 = \frac{w_0^2H^2}2 \cdot \frac{b}s $$
and so $\displaystyle w_0^2H^2=\frac{s}3<s$, hence $E$ is still $\sigma_{0,wH}-$stable for $\displaystyle w^2> \frac{s}{3H^2}$ and in particular it is stable for $w=\tilde{w}$ such that $\tilde{w}^2H^2=s$.
\end{remark}

\subsubsection{} As in the previous case, we can prove an analogous statement for the object $E(-H)[1]$ with same assumptions as in Lemma \ref{lem:stabEB} and the injectivity of the restriction map $M_{X,H}(v)\rightarrow B_C^{s+2}(2,4s)$ follows as in Theorem \ref{thm:injA}:

\begin{theorem}
\label{thm:injB}
Let $E\in M_{X,H}(v)$ be a $\mu_H-$stable vector bundle. Assume the gonality of $C$ is at least 7. Then the following statements hold true:
\begin{enumerate}[label=(\alph*)]
\item The restriction $E|_{C}$ is a slope stable vector bundle on $C$ and $h^0(C,E|_C)= 2+s$.
\item $\Hom_X(E,E(-H)[1])=0$.
\end{enumerate}

In particular the restriction map $N\rightarrow T$ is well defined and injective.
\end{theorem}

\begin{remark*}
By dropping the hypothesis on the gonalty on $C$, we still get that the restriction map is injective, because $E,E(-H)[1]$ are both stable of the same phase at $\sigma_{0,\tilde{w}}$, but we cannot infer that the restriction $E|_C$ has exactly $s+2$ global sections, what we can say about the number of global sections is that $h^0(C,E|_C)\geq s+2$.
\end{remark*}

\subsubsection{} Theorem \ref{thm:injB} proves a case of Theorem \ref{thm:BN23}.

\subsection{The case $g\equiv 2 \pmod 4$}
\label{sect:cong2}

\subsubsection{} Finally we will discuss the case $g\cong 2 \pmod 4$. We set $g=p+1$ and choose the Mukai vector $v=(4, 2H, p)$ and consider the restriction on the Brill-Noether locus $B_C^{p+4}(4,4p)$. We will prove also in this case that the restriction map $N\rightarrow T$ is injective, concluding thus the proof of Theorem \ref{thm:BN23}

\begin{lemma}
\label{lem:stabEC}
$E$ is $\sigma_{0,wH}-$stable $\forall\, \displaystyle w> \sqrt{\frac2{H^2}}$ unless $E$ has a subobject whose Mukai vector is
$$ \left({1, \frac{p-2}{2p}H+\Omega, \frac{p-1}4 }\right) \virgola $$
with $\Omega\in H^\bot$, in such case $E$ is $\sigma_{0, wH}-$stable $\forall\, \displaystyle w>\sqrt{\frac{p}{4H^2}}$.
\end{lemma}

\subsubsection{} The proof of this lemma is similar to those of Lemmas \ref{lem:stabEA} and \ref{lem:stabEB}.
As in the previous cases, the restriction map $M_{X,H}(v)\rightarrow B_C^{p+4}(4,4p)$ is well defined and injective when $g\cong 2 \pmod 4$ (so Theorem \ref{thm:BN23} holds true also in this case), moreover under the hypothesis of Lemma \ref{lem:stabEC}, the sheaves in the image of the map have no more than $p+4$ linearly independent global sections.






\section{Existence of curves with high dimensional Brill-Noether locus}

\subsubsection{} In this section we are going to give a class of examples of curves embedded in a $\Kt$ surface whose Brill-Noether locus $T=B_C^{s+2}(2,4s)$ (respectively $T=B_C^{p+4}(4,4p)$) has dimension greater than $2$, so that the restriction map $N\rightarrow T$ cannot be surjective. The main tool we will use is the injectivity proved in the previous section.

\subsubsection{} The problem of the construction of surfaces extending curves has been addressed in the literature, we report, as an example, \cite{ABS17} and \cite{CDS18}.
Let $X$ be a $\Kt$ surface and $C\hookrightarrow X$ a smooth curve of genus $g$. The \emph{Gauss-Wahl map} $\Phi_C$ is defined as
\begin{eqnarray}
\Phi_C \colon  \bigwedge^2 H^0(C,K_C) &\rightarrow &  H^0(C, 3K_C) \nonumber \\
   f \wedge g  &\mapsto  &  f\cdot dg - g \cdot df  \virgola
\end{eqnarray}
where $K_C$ is the canonical bundle on $C$. Consider, then, the adjoint map $\Phi_C^\mathbf{t}$.

\subsubsection{} Consider a curve $C$ with genus $g$ and Clifford index greater than 2 which is canonically embedded in $\P^{g-1}$.
Following the construction in \cite[\S 5]{CDS18} (based on \cite[Theorem 3]{ABS17}), for every non-zero element $v\in \ker \Phi_C^\mathbf{t}$, there exists a surface $S_v\hookrightarrow \P^g$ extending the original curve $C$; such a surface $S_v$ depends only on the class of $v$ in $\P(\ker \Phi_C^\mathbf{t})$; i.e. $S_{\lambda v}=S_v$ for $\lambda\in \C^\ast$. We are interested in smooth $\Kt$ surfaces constructed in this way.
\cite[Theorem 2.7]{CDS18} shows that if $\cork(\Phi_C)\geq 2$ and if the Clifford index of $C$ is greater than $2$, 
then the map $(X,C)\mapsto C$, where $X$ is a smooth $\Kt$ surface and $C\hookrightarrow X$ a smooth curve of genus $g$ has positive-dimensional fibre $c_g^{-1}(C)$, i.e. $C$ is contained in every surface belonging to a positive-dimensional subspace of the moduli space $\mathcal{K}_g$ of polarised $\Kt$ surfaces.

\subsubsection{} Now, if $C\hookrightarrow X$ is contained in a smooth $\Kt$ surface $X$ such that the restriction map $N_X=M_{X,H}(v)\rightarrow B_C^{s+2}(2,4s)=T$ (or $B_C^{p+4}(4,4p)$) is injective (see for instance Theorems \ref{thm:BN23} and \ref{thm:BN1}), then a deformation of $X$ in $c_g^{-1}(C)$ induces a deformation of $M_{X,H}(v)$ inside $T$. This deformation is not trivial, otherwise all the surface in an open set $c_g^{-1}(C)$ would be Fourier-Mukai partners of the same $\Kt$ surface $M_{X,H}(v)$. It follows that $\dim T \geq 3$.

\subsubsection{} Finally, let $Y$ be a Fano threefold with index $i_Y=1$, so that the generic surface in $|-K_Y|$ is a smooth $\Kt$ surface. Assume that the intersection of two such surfaces is a curve $C$ with genus $g(C)\geq 22$ and clifford index greater than 2, then \cite[Theorem 2.1]{CDS18} implies that $\cork (\Phi_C)\geq 2$ and thus $\dim T\geq 3$. We work out this construction explicitly in the following example:

\begin{example}
\label{ex:curvehBN}
Let $l_1,l_2\subseteq \P^3$ two skew lines. Let $\pi\colon Y\rightarrow \P^3$ be the blow-up of $\P^3$ at those lines. By construction, $Y$ is a Fano threefold whose anticanonical divisor is $-K_Y=4L-E_1-E_2$, where $L$ is the total transform of $\P^2\subseteq \P^3$ and $E_i=\pi^{-1}(l_i)$ is the exceptional divisor corresponding to $l_i$.  Since $K_Y$ is a primitive divisor, the index of the Fano variety is $i_Y=1$. We compute the genus of $Y$ by
$$ 2g-2=(-K_Y)^3=64L^3-48L^2(E_1+E_2)+12L(E_1^2+E_2^2)-(E_1^3+E_2^3) = 64+0-24+4=44 \virgola $$
so $g(Y)=23$. Consider a smooth $\Kt$ surface $X\in |-K_Y|$ and the ample line bundle $H = (-K_Y)|_X$, so the general curve $C\in |H|$ is a smooth curve of genus $23 \equiv 3 \pmod4$, hence a curve such that the restriction map $M_{X,H}(2,H,11)\rightarrow B_C^{13}(2,44)$ is injective.

We will show that there is a smooth curve $C\subseteq Y$ such that $C$ is contained in every surface belonging to a linear system $V\subseteq |-K_Y|$ such that surfaces in $V$ are not isomorphic to each other.

Notice that surfaces in $|-K_Y|$ are exactly strict transform of quartic surfaces in $\P^3$ containing the two lines $l_1,l_2$.

Fix two lines $r_1, r_2 \subseteq \P^3$ and consider the set $V$ of quartic surfaces containing the lines $r_1, r_2$. There is an open dense subset $U\subseteq V$ whose surfaces do not contain any further line; thus a point $X\in U$ has Picard group $\Pic(X)= L\Z \oplus r_1 \Z \oplus r_2\Z$.



There is a morphism
\begin{eqnarray}
\phi \colon U & \rightarrow & S^2(\G(1,3))\backslash \Delta_{\G(1,3)}  \nonumber \\
X & \mapsto & \{r,s\}  \virgola \nonumber
\end{eqnarray}
where $\Delta_{\G(1,3)}$ is the image of the diagonal under the quotient map $\G(1,3)\times \G(1,3)\rightarrow S^2\G(1,3)$ which sends the quartic $X$ to the set $\{r,s\}$ of lines contained in $X$. Such a morphism is surjective, because for any two lines in $\P^3$, one can construct a quartic containing those lines.

The set of quartic containing the two lines $l_1=\{w=x=0\}, l_2=\{y=z=0\}\subseteq \P^3$, i.e. the fibre of $\phi$ at the point $\{l_1,l_2\}$, is a projective space of dimension $24$.

Quartics $X_1,X_2$ in such set are isomorphic only if they are projectively equivalent.
Let $U(l_1,l_2)$ be the set of quartics in $\P^3$ containing the two lines $l_1,l_2$ and no further lines. Projectivities stabilizing $U(l_1,l_2)$ are exactly those which belong to the subgroup $G_0\subseteq \P\GL(4)$ of the projectivities which fix the two lines $l_1,l_2$ or to its coset $G_1$ of those which swaps the two lines. This defines an algebraic variety of dimension $7$.

Finally, let us call $\psi\colon U(l_1,l_2)\rightarrow \mathcal{K}_3$ the map sending any quartic in $U(l_1,l_2)$ into its isomorphism class in the moduli space of polarized $\Kt$ surfaces of genus $3$; since $\dim G = 7< \dim U(l_1,l_2)=24$, the map $\psi$ cannot be constant. More precisely $\psi(U(l_1,l_2))$ has dimension equal to $17$ and since the Picard rank of $X\in U(l_1,l_2)$ is $\rho(X)=3$, then $\psi$ is dominant onto an irreducible component of the moduli space of $\Kt$ surfaces whose Picard rank is at least $3$.

Consider two non-isomorphic smooth hyperplane sections $X_1,X_2$ of the Fano variety $Y$ constructed above such that their intersection is a smooth curve $C$. The surfaces lying in the one-dimensional linear system of hyperplane sections of $Y$ containing $C$, which coincide with the linear system of hyperplanes in $\P=\P H^0(-K_Y)$ containing a fixed two-codimensional linear subspace of $\P$, are mapped non-constantly in $\mathcal{K}_3\cap \mathcal{K}_{23}$. Thus this linear system yields a deformation of $\Kt$ surfaces inside the Brill-Noether locus $B_C^{13}(2,44)$.

This construction shows explicitely a curve $C$ where $\dim B_C^{13}(2,44)>2$.
\end{example}

\subsubsection{} The following statement summarizes the construction of the example:

\begin{corollary}
There exists a smooth curve $C$ of genus $g(C)=23$ whose Brill-Noether locus $B_C^{13}(2,44)$ has dimension greater than or equal to $3$.
\end{corollary}

\begin{remark}
Following the construction in the previous example it is possible to exhibit various examples of Brill-Noether loci of unexpected dimensions on curves whose genus is not divisible by 4.
\end{remark}

\subsection*{Acknowledgements}

We thank Professor Enrico Arbarello, for having introduced me to this topic and for several interesting discussions. We thank Professor Fabien Pazuki for many interesting discussions.


\vfill

{\textsc{University of Copenhagen, department of Mathematical Sciences, \\
Universitetparken 5, 2100 Copenhagen, Denmark}} 

\emph{E-mail address}: luigi.pagano@math.ku.dk

\end{document}